\documentclass[12pt]{amsart}
\usepackage{float}
\allowdisplaybreaks
\usepackage{longtable}
\usepackage{amsfonts,amssymb,amsmath,textcomp}
\usepackage{supertabular}
 2

\usepackage[autostyle]{csquotes}
\theoremstyle{plain}
\newtheorem{thm}{Theorem}[section]

\newtheorem{lem}[thm]{Lemma}
\newtheorem{prop}[thm]{Proposition}

\theoremstyle{definition}

\begin{document}
\baselineskip 6.1mm

\title
{Small class number fields in the family $\mathbb{Q}(\sqrt{9m^2+4m})$}
%{A class number one criterion for real quadratic fields $\mathbb{Q}(\sqrt{9m^2 \pm 4m})$}
\author{Nimish Mahapatra, Prem Prakash Pandey, Mahesh Ram}

\address[Nimish Mahapatra]{Indian Institute of Science Education and Research, Berhampur, India.}
\email{}

\address[Prem Prakash Pandey]{Indian Institute of Science Education and Research, Berhampur, India.}
\email{premp@iiserbpr.ac.in}

\address[Mahesh Ram]{Indian Institute of Science Education and Research, Berhampur, India.}
\email{}


\date{\today}


\begin{abstract}
We study the class number one problem for real quadratic fields $\mathbb{Q}(\sqrt{9m^2+ 4m})$, where $m$ is an odd integer. We show that for $m \equiv 1 \pmod 3$ there is only one such field with class number one and only one such field with class number two.
\end{abstract}

\maketitle{}

\section{Introduction}
The size of the class group of a number field $K$, called the class number of $K$, is an important object of study in algebraic number theory. The ring $\mathbb{O}_K$, of algebraic integers in $K$, is a unique factorization domain if and only if the class group of $K$ is trivial. It is known that there are only nine imaginary quadratic fields with class number one (see \cite{AB,HMS}). On the other hand, it was conjectured by Gauss that there are infinitely many real quadratic fields with class number one (see \cite{CFG}). This conjecture is still open. However, there are several partial/prelim results (see \cite{B1,B2,KKST,RM, BK1}). In particular, we mention some results showing finiteness of class number one real quadratic fields in special families of real quadratic fields.\\
% (see Biro's work on Yokoi and Chowla conjecture). In \cite{B1}, Biro proved the following theorem.
%
%\begin{thm}
%
%\end{thm}
%This was conjectured by Yokoi (see \cite{Y}). In \cite{B2}, Biro proved similar results for the following family.
%\begin{thm}

In \cite{B1}, proving a conjecture of Yokoi, Biro showed that, for any positive integer $m$ such that $m^2+4$ is square-free, the real quadratic field $\mathbb{Q}(\sqrt{m^2+4})$ has class number one precisely when $m \in \{1,3,5,7,13,17\}$. Along the similar lines, in \cite{B2} Biro proved a conjecture of Chowla; if $m$ is a positive integer such that $4m^2+1$ is square free then the real quadratic field $\mathbb{Q}(\sqrt{4m^2+1})$ has class number one if and only if $m \in \{1,2,3,5,7,13 \}$.\\

%\end{thm}
Another line of work, in these contexts, is finding necessary, and sufficient conditions for the class number of a real quadratic field to be one. We mention the work of Mollin \cite{RM} and Loboutin \cite{SL}. In \cite{BK1}, Byeon and Kim obtained one such criterion for real quadratic fields of Richaud-Degert type (R-D type); these are fields of the form $\mathbb{Q}(\sqrt{d}),$ where $d=n^2+r$ is square-free with $n,r \in \mathbb{Z}, -n<r \leq n$ and $r|4n$ (for definitions of R-D type fields in narrow sense and R-D type fields in wide sense we refer \cite{HY, HY1}). There are many results for small class number of R-D type fields (see \cite{BK1,BK2,KAM, JL}) as the fundamental units of such fields are explicitly known (\cite{GD}). \\
%In \cite{KAM}, authors obtain a conditional result on lower bound of class number of fields $\mathbb{Q}(\sqrt{n^2+1})$.\\

In this article we obtain lower bound on class number of fields which are not of R-D type. We study fields of the form $\mathbb{Q}(\sqrt{9m^2 + 4m})$, where $m$ is an odd integer. 
%extend the work of Byeon and Kim to obtain necessary and sufficient conditions for class number one for a family of real quadratic fields which are not of R-D type (make sure, as Biro calls it R-D type in 2016 paper). 
For any odd integer $m$, we put $D=9m^2+4m$ and $K=\mathbb{Q}(\sqrt{D})$. Let $\zeta_K(s)$ denote the Dedekind zeta function of the number field $K$ and $h_K$ denote its class number. We use the notation $\mathcal{N}$ to denote the number of distinct prime divisors of $m$ which are greater than $3$. We prove the following theorem in this article.\\
\begin{thm}\label{C2}
Let $m$ be any odd integer such that $m \not \equiv 2 \pmod 3$ and $D$ is square-free. Then $h_K=1$ if and only if $m \in \{-3,1,3\}$.
%the following lower bounds on class number $h_K$ holds
%\begin{center}
% 	 $h_K \geq
% 	\begin{cases}
% 	2+\mathcal{N} \qquad m \equiv 0 \pmod 3,\\
% 	3+\mathcal{N} \qquad m \equiv 1 \pmod 3,\\
% 	1+\mathcal{N} \qquad m \equiv 2 \pmod 3.
% 	\end{cases}	$
% \end{center}

%\begin{center}
%		$$\zeta_k(-1)=\frac{1}{120}
%	    (9m^3-6m^2+19m-6) $$
%	\end{center}
\end{thm}

%{\bf (Note that we do need to assume that $9m^2\pm4m$ is square-free as this does not follow from $m$ being square-free.  Take $m=13$)} \\
%Analogously we also prove the following theorem.
%\begin{thm}\label{C1}
%Let $m$ be an odd square-free integer and $K=\mathbb{Q}(\sqrt{9m^2+4m})$. If $m \not \in \{...\}$ then the following lower bounds on class number $h_K$ holds
%\begin{center}
% 	 $h_K \geq
% 	\begin{cases}
% 	2+\mathcal{N} \qquad m \equiv 0 \pmod 3,\\
% 	3+\mathcal{N} \qquad m \equiv 1 \pmod 3,\\
% 	1+\mathcal{N} \qquad m \equiv 2 \pmod 3.
% 	\end{cases}	$
% \end{center}
%\end{thm}
%From these theorems we obtain following finiteness result on the number of real quadratic fields with class number one in the family $\mathbb{Q}(\sqrt{9m^2+4m})$ and $\mathbb{Q}(\sqrt{9m^2-4m})$. 
%\begin{cor}\label{CNT2}
%Let $m$ be an odd square-free integer. Then there are no class number one fields in the family $\mathbb{Q}(\sqrt{9m^2-4m})$.
%\end{cor}
%\begin{cor}\label{CNT1}
%Let $m$ be an odd square-free integer of the form $3k+1$. Then there are no class number one fields in the family $\mathbb{Q}(\sqrt{9m^2+4m})$. 
%\end{cor}
%
%
%{/bf Check if there are some finite exceptions }\\

We remark that the case $m>0$ of Theorem \ref{C2} follows from the main theorem proved in \cite{BL}. However, our approach differs from the approach of Biro and Lapkova. From our approach, we obtain better lower bound on class number of $\mathbb{Q}(\sqrt{9m^2+4m})$ in the case $m\equiv 1 \pmod 3$. To this effect we prove the following theorem.
%When $m \equiv 1 \pmod 3$, then there are only finitely many fields of the form $\mathbb{Q}(\sqrt{9m^2+4m})$ with class number $2$.
\begin{thm}\label{C3}
Let $m \equiv 1 \pmod 3$ be such that $D=9m^2+4m$ is square-free. If $m \not \in \{-5,1\}$ then $h_K \geq 3$.
\end{thm}

In case $m$ has large number of prime factors then the class number $h_K$ is large. We record this in the following theorem.
\begin{thm}\label{C4}
If $m$ has three or more prime factors and $D$ is square-free then 
\begin{center}
 	 $h_K \geq
 	\begin{cases}
 	1+\mathcal{N} \qquad if\quad m \equiv 2 \pmod 3,\\
 	2+\mathcal{N} \qquad otherwise,\\
 	\end{cases}	$
 \end{center}
 where $\mathcal{N}$ is number of prime factors of $m$ bigger than $3$.
\end{thm}

%{\bf If these are all the results we can prove then we give an outline of the paper now} I was trying to prove that the polynomial $X^2-3mX-m$ takes certain composite/prime values with $X \in \{1, \ldots, m-1\}$ but as of now there is some lacuna.\\

For integers $d$ and $N$ the Diophantine equation $x^2-dy^2=N$ is widely studied, under the popular name `Pell$'$s equations'. The necessary and sufficient conditions for the solubility of such equations are known for long. We refer the article of Mollin \cite{RM1} and the references therein. Employing the techniques used in the study of the above problems, we present a family of Pell's equations and prove that they are always soluble.

\begin{thm}\label{C5}
Let $m$ be an odd positive square-free integer such that $q=9m+4$ is a prime. Then for $D=m(9m+4)$, the Diophantine equation $x^2-Dy^2=4q$ is always soluble.
\end{thm}

In Section 2 we recall some results on special values of Dedekind zeta function and develop preliminaries required for the proofs. In section 3 we give all the proofs. Section 4 presents some computations performed by us with the help of Pari.\\

\section{Preliminaries}
We begin this section by mentioning two methods to compute special values of Dedekind zeta function associated to a real quadratic field. By specializing Siegel's formula (see \cite{CS}) for $\zeta_K(1-2n)$, in \cite{DZ}, Zagier obtained the value of $\zeta_K(-1)$ for any real quadratic field. We mention this in the following theorem.

\begin{thm}\label{DZ}
(Zagier \cite{DZ}) Let $K$ be a real quadratic field with discriminant $D$. Then 
$$ \zeta_K(-1)=\frac{1}{60} \sum_{\substack{ |t|< \sqrt{D} \\t^2 \equiv D\ (\bmod 4)}} \sigma\left( \frac{D-t^2}{4}\right),   $$
where $\sigma(n)$ denotes the sum of divisors of $n$.
\end{thm}

However, there is another method, due to H. Lang (see \cite{HL}), for computing special values of partial Dedekind zeta function of an ideal class in real quadratic fields. Let $K=\mathbb{Q}(\sqrt{d})$ be a real quadratic field with discriminant $D$ and fundamental unit $\epsilon_d$. Let $\mathfrak{A}$ be an ideal class in the class group of $K$, $\mathfrak{a}$ be an integral ideal in the class $\mathfrak{A}^{-1}$ with an integral basis $\{r_1,r_2\}$. Let $r_1',r_2'$ denote the conjugate of $r_1,r_2$ respectively and let $$\delta(\mathfrak{a})=r_1r_2'-r_1'r_2.$$ 
Let $M=$ $\begin{bmatrix} a & b \\ c & d \end{bmatrix}$ be an integral matrix satisfying
$$ \epsilon \begin{bmatrix}
r_1 \\ r_2
\end{bmatrix}= M  \begin{bmatrix}
r_1 \\ r_2
\end{bmatrix}.$$

Then the partial zeta value $\zeta_K(-1,\mathfrak{A})$, associated to the ideal class $\mathfrak{A}$, was obtained by H. Lang (see \cite{HL}) in terms of the parameters $\delta(\mathfrak{a}),a,b,c,d$ and generalized Dedekind sums.

\begin{thm}\label{HL}
(Lang \cite{HL}) With the above notations, we have \\
\begin{multline}\label{HLF}
\zeta_K(-1,\mathfrak{A})=\frac{sgn\delta(\mathfrak{a})r_2r_2'}{360N(\mathfrak{a})c^3}\{(a+d)^3-6(a+d)N(\epsilon)-240c^3(sgnc)S^3(a,c)\\+180ac^3(sgnc)S^2(a,c)-240c^3(sgnc)S^3(d,c)+180dc^3(sgnc)S^2(d,c)\},
\end{multline}
where $N(\mathfrak{a})$ is the norm of the ideal $\mathfrak{a}$, and $S^i(a,c)=S^i_4(a,c)$ denote the generalised Dedekind sum.
\end{thm}

%To determine the value of $a,b,c,d$, Byeon and Kim (see \cite{BK1}) proved the following lemma. 
%{\bf It seems we do not require this lemma}
%
%\begin{lem}\label{L1}
%The entries of $M$ are given by
%	\begin{align*}
%	a=Tr\left(\frac{r_1r_2'\epsilon}{\delta(\mathfrak{a})}\right), b=Tr\left(\frac{r_1r_1'\epsilon'}{\delta(\mathfrak{a})}\right),
%	c=Tr\left(\frac{r_2r_2'\epsilon}{\delta(\mathfrak{a})}\right),
%	d=Tr\left(\frac{r_1r_2'\epsilon'}{\delta(\mathfrak{a})}\right).
%	\end{align*} 
%Furthermore, det$(M)=N(\epsilon)$ and  $bc\neq0$.
%\end{lem}

Now we mention some special values of generalized Dedekind sum as obtained in \cite{HK}.

\begin{lem}\label{L2}
For any positive integer $m$, we have
\begin{enumerate}
\item $S^3(\pm1,m)=\pm(-m^4+5m^2-4)/(120m^3)$
\item $S^2(\pm1,m)=(m^4+10m^2-6)/(180m^3)$
\end{enumerate}
\end{lem}

%{\bf Do we need the following lemma?}
%\begin{lem}\label{L3}
%For any positive even integer $m$, we have
%\begin{enumerate}
%\item $S^3(m\pm1,2m)=\pm S^1(m+1,2m)=\mp(m^4-50m^2+4)/(960m^3)$
%\item $S^2(m-1,2m)=S^2(m+1,2m)=(m^4+100m^2-6)/(1440m^3)$
%\item $S^3(m+1,4m)=(-m^4-180m^3+410m^2-4)/(7680m^3)$
%\item $S^3(m-1,4m)=(m^4-180m^3-410m^2+4)/(7680m^3)$
%\item $S^2(m-1,4m)=S^2(m+1,4m)=(m^4+820m^2-6)/(11520m^3)$
%\end{enumerate}
%\end{lem}

For other values of generalized Dedekind sum we recall the formula obtained by Lang. For $n=2,3,4, \ldots, $ and $r=0,1, \ldots, 2n$, Lang \cite{HL} obtained following formula 
$$S^r_{2n}(a,c)=\sum_{ j \pmod c}P_{2n-r}(\frac{j}{c})P_r(\frac{aj}{c}).$$
Here $P_t(x)=\sum_{s=0}^t \binom{t}{s}B_s (\{x\})^{(t-s)}$, $B_s$ is the $s^{th}$ Bernoulli number and $\{x\}$ denotes the fractional part for any real number $x$. In particular we have $P_1(x)=\{x\}-\frac{1}{2}$, $P_2(x)=\{x\}^2-\{x\}+\frac{1}{6}$ and $P_3(x)=\{x\}^3-\frac{3}{2}\{x\}^2+\frac{1}{2}\{x\}$. The generalized Dedekind sums we are concerned with are recorded in the following lemma.

\begin{lem}\label{L4}
We have
$$S^2(a,c)=\sum_{ j \pmod c}P_2(\frac{j}{c})P_2(\frac{aj}{c}) \qquad \mbox{ and }$$
$$S^3(a,c)=\sum_{ j \pmod c}P_1(\frac{j}{c})P_3(\frac{aj}{c}).$$
\end{lem}

We now mention a result of Yokoi that facilitates us with the fundamental unit of the fields of the form $\mathbb{Q}(\sqrt{9m^2+ 4m})$.

\begin{thm}\label{FU}
	(Lemma 3 in Yokoi \cite{HY}) Let $p$ be any prime congruent to -1 $mod$ 4, and assume that an unit $\epsilon$ of a real quadratic field $\mathbb{Q}(\sqrt{D})$ ($D>0$ square-free) is of the form
	\begin{equation*}
		\epsilon=\frac{1}{2}(t+p\sqrt{D}) \quad or \quad t+p\sqrt{D} \quad(t>0)
	\end{equation*}
	Then, the real quadratic field $\mathbb{Q}(\sqrt{D})$ is of R-D type in narrow sense or the unit $\epsilon$ is the fundamental unit of $\mathbb{Q}(\sqrt{D})$ satisfying $N(\epsilon)=1$.
\end{thm}

%{\bf Do we need the following theorem?}
\begin{thm}\label{FUM}
	(Yokoi \cite{HY}) For any prime $p$ congruent to $-1$ modulo $4$, there exists an integer $D_0=D_0(p)$ such that if $D=p^2m^2\pm4m$ has no square factor except $4$ and is bigger than $D_0$, then the fundamental unit $\epsilon_D$ of the real quadratic field $\mathbb{Q}(\sqrt{D})$ is of the following form:\\
 \begin{center}
 	 $\epsilon_D=
 	\begin{cases}
 	\frac{1}{2}[(p^2m\pm2)+p\sqrt{p^2m^2\pm4m}] \qquad D: square-free,\\
 	\frac{1}{2}(p^2m\pm2)+p\sqrt{\frac{1}{4}(p^2m^2\pm4m)}] \quad otherwise,\\
 	\end{cases}	$
 \end{center}
and $N(\epsilon)=1$.
\end{thm}

We end this section with a result of Hasse (see \cite{HH}).
\begin{thm}\label{HH}
	(Hasse \cite{HH}) If a real quadratic field $K=\mathbb{Q}(\sqrt{D})$ has $h_K=1$ then $D=p$, $2p$ or $qr$ where $p$, $q$ and $r$ are primes and $p \equiv 1 \pmod 4, q\equiv r\equiv3 (\bmod4)$.
\end{thm}

\section{Proofs}
%We shall give proof of Theorem \ref{C1}, proof of Theorem \ref{C2} follows along the same line. 
We note that $D=9m^2+4m$ is square-free and $D \equiv 1 \pmod 4$. So the discriminant of the field $K=\mathbb{Q}(\sqrt{D})$ is $D$. We provide proofs for $m>0$. Proof for $m<0$ goes along the same line, however it is not identical. This difference creeps in because of the difference of fundamental units in two cases. As in Theorem \ref{FUM}, let $\epsilon_D=\frac{1}{2}[(9m+2)+3\sqrt{D}]$. Then $N(\epsilon_D)=1$ and $\epsilon_D$ is a unit of $K$. It is easily observed that there are no integers $n,r$ with $-n< r \leq n$ and $r|4n$ for which the equality $9m^2+4m=n^2+r$ holds. By Theorem \ref{FU}, $\epsilon_D$ is the fundamental unit of $K$. When $m<0$ then $D=9m'^2-4m'$ for $m'=-m>0$. By Theorem \ref{FU} the fundamental unit in this case is given by $\epsilon_D=\frac{1}{2}[(9m'-2)+3\sqrt{D}]$.\\

Once again, we emphasize that now onwards $m>0$. We shall use notations $\mathfrak{C}, \mathfrak{U}$ to respectively denote the trivial ideal class and the ideal class containing the prime ideal above $3$. If $p>3$ is a prime divisor of $m$, then $\mathfrak{P}$ shall denote the ideal class containing the prime ideal above $p$. In case we are having two distinct prime divisors $p_1, p_2>3$ of $m$ then the corresponding ideal classes are denoted by $\mathfrak{P}_1, \mathfrak{P}_2$.\\

The following lemma is an easy consequence of the factorization of rational primes in quadratic fields.
\begin{lem}\label{IB}
(i) If $m \equiv 0 \pmod 3$ then $3 \mathbb{O}_K=<3,\frac{3+\sqrt{D}}{2}>^2$ and $\{3,\frac{3+\sqrt{D}}{2}\} $ is an integral basis of the ideal $<3,\frac{3+\sqrt{D}}{2}>$.\\
(ii) If $m \equiv 1 \pmod 3$ then $3 \mathbb{O}_K=<3,\frac{1+\sqrt{D}}{2}> <3,\frac{1-\sqrt{D}}{2}>$ and $\{3,\frac{1+\sqrt{D}}{2}\} $ is an integral basis of the ideal $<3,\frac{1+\sqrt{D}}{2}>$.\\
(iii) If $m \equiv 2 \pmod 3$ then $3 \mathbb{O}_K=<3,\frac{3+3\sqrt{D}}{2}>$ and $\{3,\frac{3+3\sqrt{D}}{2}\} $ is an integral basis of the ideal $<3,\frac{3+3\sqrt{D}}{2}>$.\\
(iv) If $p>3$ is a prime divisor of $m$ then $p \mathbb{O}_K=<p,\frac{p+\sqrt{D}}{2}>^2$ and $\{p,\frac{p+\sqrt{D}}{2}\} $ is an integral basis of the ideal $<p,\frac{p+\sqrt{D}}{2}>$.\\
(v) If $q>3$ is a prime divisor of $D$ then $q \mathbb{O}_K=<q,\frac{q+\sqrt{D}}{2}>^2$ and $\{q,\frac{q+\sqrt{D}}{2}\} $ is an integral basis of the ideal $<q,\frac{q+\sqrt{D}}{2}>$.\\
\end{lem}

\begin{prop}\label{P1}
Let $\mathfrak{C}$ be the trivial ideal class in the class group of $K$. Then 
$$\zeta_K(-1,\mathfrak{C})=\frac{1}{120}(9m^3+6m^2+19m+6).$$
\end{prop}
\begin{proof}
We consider the integral ideal $\mathbb{O}_K$ in the ideal class $\mathfrak{C}^{-1}$. Then $r_1=\frac{1+\sqrt{D}}{2},r_2=1$ is an integral basis of $\mathbb{O}_K$. If $M=$ $\begin{bmatrix} a & b \\ c & d \end{bmatrix}$ is an integral matrix satisfying
$$ \epsilon_D \begin{bmatrix}
r_1 \\ r_2
\end{bmatrix}= M  \begin{bmatrix}
r_1 \\ r_2
\end{bmatrix}.$$
Then we obtain
$$a=\frac{1}{2}[(9m+2)+3], b=\frac{3}{4}(9m^2+4m-1), c=3 \mbox{ and } d=\frac{1}{2}[(9m+2)-3].$$
Clearly $a=1 \pmod c$ and $d=1 \pmod c$.\\
From the Lemma \ref{L2} we get
$$S^3(a,c)=S^3(d,c)=S^3(1,3)=\frac{-1}{81}$$
and
$$S^2(a,c)=S^2(d,c)=S^2(1,3)=\frac{11}{324}.$$
Substituting these in Theorem \ref{HL} we obtain the proposition.
\end{proof}

%Certainly, $\zeta_K(-1) \geq \zeta_K(-1, \mathfrak{C})$ and equality holds if and only if $h_K=1$ (see \cite{BK1}). From this Theorem \ref{C1} follows immediately.

\begin{prop}\label{P2}
If $3|m$ and $\mathfrak{U}$ is an ideal class in the class group of the field $K$ such that $<3,\frac{3+\sqrt{D}}{2}> \in \mathfrak{U}^{-1}$ then
$$\zeta_K(-1, \mathfrak{U})=\frac{1}{360}(3m^3+2m^2+273m+162).$$
\end{prop}
\begin{proof}
By Lemma \ref{IB}, $r_1=\frac{3+\sqrt{D}}{2},r_2=3$ is an integral basis of the ideal $\mathfrak{u}=<3,\frac{3+\sqrt{D}}{2}>$. We have $\delta(\mathfrak{u})=3\sqrt{D}$.\\
If $M=$ $\begin{bmatrix} a & b \\ c & d \end{bmatrix}$ is an integral matrix satisfying
$$ \epsilon_D \begin{bmatrix}
r_1 \\ r_2
\end{bmatrix}= M  \begin{bmatrix}
r_1 \\ r_2
\end{bmatrix}.$$
Then we obtain $a=\frac{9m+11}{2} , b=\frac{9m^2+4m-9}{4} ,c=9 , d=\frac{9m-7}{2} $. From this it is clear that $a \equiv 1 \pmod c, d \equiv 1 \pmod c$. Hence, by definition, we have $S^3(a,c)=S^3(d,c)=S^3(1,c)$ and $S^2(a,c)=S^2(d,c)=S^2(1,c)$. By Lemma \ref{L2}, we obtain
$$S^3(a,c)=S^3(d,c)=\frac{-154}{3.9^3} \quad \mbox{ and } S^2(a,c)=S^2(d,c)=\frac{491}{12.9^3}.$$
Substituting these in (\ref{HLF}) gives
$$\zeta_K(-1, \mathfrak{U})=\frac{1}{360}(3m^3+2m^2+273m+162).$$
\end{proof}
%{\bf The above proposition is covered in proposition \ref{P5}, and shall be deleted.}

\begin{prop}\label{P3}
If $m \equiv 1 \pmod 3$ and $\mathfrak{U}$ is an ideal class in the class group of the field $K$ such that $<3,\frac{1+\sqrt{D}}{2}> \in \mathfrak{U}^{-1}$ then
$$\zeta_K(-1, \mathfrak{U})=\frac{1}{360}(3m^3+2m^2+113m+2).$$
\end{prop}
\begin{proof}
By Lemma \ref{IB}, $r_1=\frac{1+\sqrt{D}}{2},r_2=3$ is an integral basis of the ideal $\mathfrak{u}=<3,\frac{1+\sqrt{D}}{2}>$. We have $\delta(\mathfrak{u})=3\sqrt{D}$.\\
If $M=$ $\begin{bmatrix} a & b \\ c & d \end{bmatrix}$ is an integral matrix satisfying
$$ \epsilon_D \begin{bmatrix}
r_1 \\ r_2
\end{bmatrix}= M  \begin{bmatrix}
r_1 \\ r_2
\end{bmatrix}.$$
Then we obtain $a=\frac{9m+5}{2} , b=\frac{9m^2+4m-1}{4} ,c=9 , d=\frac{9m-1}{2} $. From this it is clear that $a \equiv 7 \pmod c, d \equiv 4 \pmod c$. Hence, by definition, we have $S^3(a,c)=S^3(7,9)$, $S^3(d,c)=S^3(4,9)$ and $S^2(a,c)=S^2(7,9), S^2(d,c)=S^2(4,9)$. By Lemma \ref{L4}, we obtain
$$S^3(7,9)=\frac{62}{3.9^3}, \quad S^3(4,9)=\frac{8}{3.9^3}, \quad S^2(7,9)=\frac{203}{12.9^3} \mbox{ and } \quad S^2(4,9)=\frac{203}{12.9^3}.$$
Substituting these in (\ref{HLF}) gives
$$\zeta_K(-1, \mathfrak{U})=\frac{1}{360}(3m^3+2m^2+113m+2).$$
\end{proof}

%\begin{prop}\label{P4}
%If $m \equiv 2 \pmod 3$ and $\mathfrak{U}$ is an ideal class in the class group of the field $K$ such that $<3,\frac{3+3\sqrt{D}}{2}> \in \mathfrak{U}^{-1}$ then
%$$\zeta_K(-1, \mathfrak{U})=\frac{1}{9 \times 120}(9m^3+6m^2+19m+6).$$
%\end{prop}
%\begin{proof}
%By Lemma \ref{IB}, $r_1=\frac{3+3\sqrt{D}}{2},r_2=3$ is an integral basis of the ideal $\mathfrak{u}=<3,\frac{3+3\sqrt{D}}{2}>$. We have $\delta(\mathfrak{u})=9\sqrt{D}$.\\
%If $M=$ $\begin{bmatrix} a & b \\ c & d \end{bmatrix}$ is an integral matrix satisfying
%$$ \epsilon_D \begin{bmatrix}
%r_1 \\ r_2
%\end{bmatrix}= M  \begin{bmatrix}
%r_1 \\ r_2
%\end{bmatrix}.$$
%Then we obtain $a=\frac{9m+5}{2} , b=\frac{3D-3}{4} ,c=3 , d=\frac{9m-1}{2} $. Thus, all the parameters are same as in the Proposition \ref{P1}, except $N(\mathfrak{u})=9$. Consequently we obtain the proposition.
%\end{proof}

\begin{prop}\label{P5}
Let $p$ be a prime dividing $m$. If $\mathfrak{P}$ is an ideal class in the class group of the field $K$ such that $<p,\frac{p+\sqrt{D}}{2}> \in \mathfrak{P}^{-1}$ then
$$\zeta_K(-1, \mathfrak{P})=\frac{1}{120\times p^2}(9m^3+6m^2+9mp^4+10mp^2+6p^4).$$
\end{prop}
\begin{proof}
By Lemma \ref{IB}, $r_1=\frac{p+\sqrt{D}}{2},r_2=p$ is an integral basis of the ideal $\mathfrak{p}=<p,\frac{p+\sqrt{D}}{2}>$. We have $\delta(\mathfrak{u})=p\sqrt{D}$.\\
If $M=$ $\begin{bmatrix} a & b \\ c & d \end{bmatrix}$ is an integral matrix satisfying
$$ \epsilon_D \begin{bmatrix}
r_1 \\ r_2
\end{bmatrix}= M  \begin{bmatrix}
r_1 \\ r_2
\end{bmatrix}.$$
Then we obtain $a=\frac{9m+2+3p}{2} , b=\frac{3D-3p^2}{4p} ,c=3p , d=\frac{9m+2-3p}{2} $. Clearly $a \equiv 1 \pmod c$ and $d \equiv 1 \pmod c$. Consequently,
$$S^3(a,c)=S^3(d,c)=S^3(1,3p)=\frac{-(3p)^4+5(3p)^2-4}{120(3p)^3}$$
and
$$S^2(a,c)=S^2(d,c)=S^2(1,3p)=\frac{(3p)^4+10(3p)^2-6}{180(3p)^3}.$$
Substituting these in (\ref{HLF}) gives
$$\zeta_K(-1, \mathfrak{P})=\frac{1}{120\times p^2}(9m^3+6m^2+9mp^4+10mp^2+6p^4).$$
\end{proof}

\begin{proof}
(Theorem \ref{C2}) If $m= 1$ then $K=\mathbb{Q}(\sqrt{13})$ and its class number is one. Now we assume that $m>1$. If $m$ is not a prime then by Theorem \ref{HH} it follows that $h_K>1$. Thus we can assume that $m=p$ for some odd prime $p$. 

From Proposition \ref{P1}, Proposition \ref{P2}, Proposition \ref{P3}, we see that 
$$\zeta_K(-1,\mathfrak{C}) \neq \zeta_K(-1,\mathfrak{U}), \quad \mbox{ unless } m=3.$$
%{\bf Shouldn't we have $\zeta_K(-1,\mathfrak{C}) = \zeta_K(-1,\mathfrak{U})$ for $m=3$? But our computations do not agree with this, check the mistake.}
This proves that $h_K\geq 2$ whenever $m>3$. For $m=3$, we have $K=\mathbb{Q}(\sqrt{93})$ and $h_K=1$. This proves the theorem.
%Now either $m \equiv 1 \pmod 3$ or $m\equiv 2 \pmod 3$. Let $\mathfrak{C}$ be the trivial ideal class in the class group of $K$ and $\mathfrak{U}$ be the ideal class in $K$ containing the prime above $3$. Using Proposition \ref{P1}, Proposition \ref{P3} and Proposition \ref{P4} we see that
%$$\zeta_K(-1,\mathfrak{C}) \neq \zeta_K(-1,\mathfrak{U}).$$
%As a consequence the ideal classes $\mathfrak{C}$ and $\mathfrak{U}$ are different and we have $h_K \geq 2$.
\end{proof}

%If we assume that $m$ has atleast $3$ prime factors, then we obtain following lower bound on the class number $h_K$. \begin{cor}
%The following lower bound on the class number $h_K$ holds.
%$$h_K \geq 1+\mathcal{N}.$$
%Here $\mathcal{N}$ is the number of distinct prime factors of $m$.
%\end{cor}
%
%\begin{proof}
%
%\end{proof}
%
%
%
%
%
%
%
%
%
%\begin{prop}
%Let $m$ be an odd square-free integer and $D=9m^2+4m$. Then we have 
%\begin{equation}\label{E1}
%\sum_{\substack{ |t|< \sqrt{D} \\t^2 \equiv D\ (\bmod 4)}} \sigma\left( \frac{D-t^2}{4}\right)\geq \frac{1}{2}(9m^3+6m^2+19m+6),  
%\end{equation}
%where $\sigma(n)$ denotes the sum of divisors of $n$.
%\end{prop}
\begin{proof}
(Theorem \ref{C3})
We have already seen that the ideal classes $\mathfrak{C}$ and $\mathfrak{U}$ are different elements in the class group of $K$. To establish Theorem \ref{C3}, we obtain a lower bound on 
$$
\sum_{\substack{ |t|< \sqrt{D} \\t^2 \equiv D\ (\bmod 4)}} \sigma\left( \frac{D-t^2}{4}\right).
$$
Since
	\begin{align*}
		(&3m)^2<D<(3m+1)^2,
	\end{align*}
we have
\begin{align*}
3m=\lfloor\sqrt{D}\rfloor<\sqrt{D}<(3m+1).
	\end{align*}
Further $t^2\equiv D (\bmod4) $ holds if and only if $t$ is odd. Thus, we conclude that $t$ runs over integers
	\begin{center}
			$ t=3m-2s$, for  $s=0,1,2,...,\left(\frac{3m-1}{2}\right)$.
	\end{center}
For $t=3m-2s$ we have $\left(\frac{D-t^2}{4}\right)=-s^2+3ms+m$.	 Thus
\begin{equation}\label{E1}
\sum_{\substack{ |t|< \sqrt{D} \\t^2 \equiv D\ (\bmod 4)}} \sigma\left( \frac{D-t^2}{4}\right)=2\left(\sum_{\substack{ s=0}}^ {\frac{(3m-1)}{2}} \sigma(-s^2+3ms+m)\right).
\end{equation}
For each $s$ considering the trivial divisors of $-s^2+3ms+m$ we obtain following contribution in the right side of (\ref{E1})
%To obtain lower bound on $$
%\sum_{\substack{ |t|< \sqrt{D} \\t^2 \equiv D\ (\bmod 4)}} \sigma\left( \frac{D-t^2}{4}\right)
%$$
%we need lower bounds on $\sigma(-s^2+3ms+m)$. For each $s$,  and by considering trivial divisors we obtain,
%\begin{equation*}
%\sigma\left( \frac{D-t^2}{4}\right)\geq (-s^2+3ms+m)+1  \quad \forall s=0,1,...,\frac{3m-1}{2} .
%\end{equation*}
%Thus, the trivial divisors contribute
\begin{equation}\label{E2}
2\left(\sum_{\substack{s=0}}^{\frac{3m-1}{2}}1+m(1+3s)-s^2\right)=\frac{1}{2}(9m^3+6m^2+7m+2).
\end{equation}

Since $m \equiv 1 \pmod 3$, we see that $3$ is a divisor of $-s^2+3ms+m$ whenever $s\equiv 1 \pmod 3$ or $s \equiv 2 \pmod 3$. Thus, for each $t=3m-2s$ with $s\equiv 1 \pmod 3$ or $s \equiv 2 \pmod 3$ we get extra contribution of $3+\frac{-s^2+3ms+m}{3}$ in the right side of (\ref{E1}). The total contribution due to these factors, in the right side of (\ref{E1}) is\\
\begin{equation}\label{E3}
2\left(\sum_{s=0, 3 \nmid s}^{\frac{3m-1}{2}}3+\frac{m(1+3s)-s^2}{3}\right)=m^3+\frac{2m^2}{3}+\frac{19m}{3}.
\end{equation}
Further, for $s=m$ and $t=m$ we get extra contribution of $m+(2m+1)$ in the right side of (\ref{E1}).\\ 
%\begin{equation}\label{E3}
%\sigma\left( \frac{D-t^2}{4}\right)\geq \sigma(-m^2+3m^2+m)  \geq  (-m^2+3m^2+m)+1+m+(2m+1) .
%\end{equation}
Taking in account all these contributions we get
\begin{equation}\label{E4}
	\sum_{\substack{ |t|< \sqrt{D} \\t^2 \equiv D\ (\bmod 4)}} \sigma\left( \frac{D-t^2}{4}\right)\geq 
	\frac{1}{2}(9m^3+6m^2+7m+2)+(m^3+\frac{2m^2}{3}+\frac{19m}{3})+2(3m+1).
\end{equation}
From Theorem \ref{DZ}, we get 
\begin{equation}\label{E5}
\zeta_K(-1) \geq \frac{1}{60}\left( \frac{1}{2}(9m^3+6m^2+13m+4)+(m^3+\frac{2m^2}{3}+\frac{19m}{3})+2(3m+1) \right)
\end{equation}
Using Proposition \ref{P1} and Proposition \ref{P3}, we see that
$$\zeta_K(-1)>\zeta_K(-1,\mathfrak{C})+\zeta_K(-1,\mathfrak{U}),$$
whenever $m>4$. Thus $h_K \geq 3$ whenever $m> 4$. 
\end{proof}

\begin{proof}
(Theorem \ref{C4}) As noted in the proof of Theorem \ref{C2}, the ideal classes $\mathfrak{C}$ and $\mathfrak{U}$ are two distinct elements in the class group of $K$. Let $p>3$ be a prime divisor of $m$. Using Proposition \ref{P1} and Proposition \ref{P2} we note that $\zeta_K(-1,\mathfrak{C})=\zeta_K(-1, \mathfrak{P})$ gives
$$p^2(9m^3+6m^2+19m+6)=(9m^3+6m^2+9mp^4+10mp^2+6p^4).$$
Simplifying this we obtain $$(9m^3+6m^2-9mp^2-6p^2)(p^2-1)=0.$$
This gives $9m^3+6m^2-9mp^2-6p^2=0$ which leads to $m=p$. But this is not possible as $m$ has at least three distinct prime factors. Consequently the ideal classes $\mathfrak{C}$ and $\mathfrak{P}$ are different.\\

Similarly we see that $\zeta_K(-1,\mathfrak{U}) \neq \zeta_K(-1,\mathfrak{P})$. Further, if $p_1$, $p_2$ are two distinct primes dividing $m$ then we see that $\zeta_K(-1,\mathfrak{P}_1) \neq \zeta_K(-1,\mathfrak{P}_2)$.\\

Thus we get a different ideal class in the class group of $K$ for each prime $p>3$ dividing $m$ in addition to ideal classes $\mathfrak{C}$ and $\mathfrak{U}$. This proves the theorem.

% then we claim that the ideal class $\mathfrak{P}$ is different than $\mathfrak{C},\mathfrak{U} $. 
%
%Let $p_1, p_2>3$ be two distinct prime divisors of $m$ and let $\mathfrak{P}_1, \mathfrak{P}_2$ be the ideal classes containing the prime ideals above $p_1, p_2$ respectively. We claim that $\mathfrak{P}_1 \neq \mathfrak{P}_2$. If we assume that the opposite holds, then we obtain $\zeta_K(-1,\mathfrak{P}_1)=\zeta_K(-1,\mathfrak{P}_2)$. From Proposition \ref{P5}, this gives
%$$\frac{1}{p_1^2}(729m^3+486m^2+729mp_1^4+810mp_1^2+486p_1^4)=\frac{1}{p_2^2}(729m^3+486m^2+729mp_2^4+810mp_2^2+486p_2^4).$$
\end{proof}

Before proving Theorem \ref{C5}, we need one more proposition about zeta values. We recall that $D=m(9m+4)$ and consider the case when $9m+4=q$ is an odd prime. Let $\mathfrak{q}$ denote the prime ideal in $K$ lying above $q$ and $\mathfrak{Q}$ be an ideal class such in the class group of $K$ such that $\mathfrak{q}$ lies in $\mathfrak{Q}^{-1}$.

\begin{prop}\label{P6}
Let $D=mq$, where $q=9m+4$ is an odd prime. With above notations,
$$\zeta_K(-1,\mathfrak{Q})=\frac{1}{120}(9m^3+6m^2+19m+6).$$
\end{prop}
\begin{proof}
By Lemma \ref{IB}, $r_1=\frac{q+\sqrt{D}}{2},r_2=q$ is an integral basis of the ideal $\mathfrak{q}=<q,\frac{q+\sqrt{D}}{2}>$. We have $\delta(\mathfrak{q})=q\sqrt{D}$.\\
If $M=$ $\begin{bmatrix} a & b \\ c & d \end{bmatrix}$ is an integral matrix satisfying
$$ \epsilon_D \begin{bmatrix}
r_1 \\ r_2
\end{bmatrix}= M  \begin{bmatrix}
r_1 \\ r_2
\end{bmatrix}.$$
Then we obtain $a=\frac{9m+2+3q}{2} , b=\frac{3D-3q^2}{4q} ,c=3q , d=\frac{9m+2-3q}{2} $. Clearly $a \equiv 2q-1 \pmod c$ and $d \equiv 2q-1 \pmod c$. Consequently,
$$S^3(a,c)=S^3(d,c)=S^3(2q-1,3q), \quad S^2(a,c)=S^2(d,c)=S^2(2q-1,3q).$$
Now we compute $S^2(2q-1,3q)$ using Lemma \ref{L4}.
\begin{multline}\label{P6E1}
S^2(2q-1,3q)=\sum_{j=0}^{3q-1}P_2\left(\frac{j}{3q}\right)P_2\left(\frac{(2q-1)j}{3q}\right)\\
=\sum_{k=0}^{q-1}P_2\left(\frac{3k}{3q}\right)P_2\left(\frac{(2q-1)(3k)}{3q}\right)+\sum_{k=0}^{q-1}P_2\left(\frac{3k+1}{3q}\right)P_2\left(\frac{(2q-1)(3k+1)}{3q}\right)\\
+\sum_{k=0}^{q-1}P_2\left(\frac{3k+2}{3q}\right)P_2\left(\frac{(2q-1)(3k+2)}{3q}\right)
\end{multline}
Since $q \equiv 1 \pmod 3$, we observe that
\begin{multline*}
\{(3k)(2q-1)(3q)\}=(3q-3k)/3q \mbox{ for } k>0, \quad \{(3k+1)(2q-1)/3q\}=(-3k+2q-1)/3q \quad\\
 \mbox{ for }k \leq (2q-2)/3, \quad
 \{(3k+1)(2q-1)/3q\}=(-3k+5q-1)/3q \mbox{ for }k \geq (2q+1)/3,\\
  \{(3k+2)(2q-1)/3q\}=(-3k+q-2)/3q \mbox{ for } k \leq (q-4)/3\\
 \mbox{ and } \{(3k+2)(2q-1)/3q\}=(-3k+4q-2)/3q \mbox{ for } k \geq (q-1)/3.
\end{multline*}
Recall that $P_i(x)=P_i(\{x\})$ for $i=1,2,3$ and $x \in \mathbb{R}$. Hence we obtain
\begin{multline}\label{ES2}
S^2(2q-1,3q)=\sum_{k=0}^{q-1} P_2\left(\frac{3k}{3q}\right)P_2\left(\frac{3q-3k}{3q}\right)+\sum_{k=0}^{\frac{2q-2}{3}}P_2\left(\frac{3k+1}{3q}\right)P_2\left(\frac{-3k+2q-1}{3q}\right)\\
+\sum_{k=\frac{2q+1}{3}}^{q-1}P_2\left(\frac{3k+1}{3q}\right)P_2\left(\frac{-3k+5q-1}{3q}\right)+\sum_{k=0}^{\frac{q-4}{3}}P_2\left(\frac{3k+2}{3q}\right)P_2\left(\frac{-3k+q-2}{3q}\right)\\
+\sum_{k=\frac{q-1}{3}}^{q-1}P_2\left(\frac{3k+2}{3q}\right)P_2\left(\frac{-3k+4q-2}{3q}\right).\\
\end{multline}
Now, we evaluate each sum individually.
\begin{multline*}
\sum_{k=0}^{q-1} P_2\left(\frac{3k}{3q}\right)P_2\left(\frac{3q-3k}{3q}\right)=\sum_{k=0}^{q-1}\left(\frac{q^2-6kq+6k^2}{6q^2}\right) \left( \frac{q^2-6kq+6k^2}{6q^2}\right)\\
=\sum_{k=0}^{q-1}\left( \frac{q^4-12kq^3+48k^2q^2-72k^3q+36k^4}{36q^4}\right).
\end{multline*}
Now we use the formula for sum of fixed powers of consecutive numbers to obtain
\begin{equation}\label{ES21}
\sum_{k=0}^{q-1} P_2\left(\frac{3k}{3q}\right)P_2\left(\frac{3q-3k}{3q}\right)=\frac{q^5+10q^3-6q}{180q^4}.
\end{equation}
Next, we have
\begin{multline*}
\sum_{k=0}^{\frac{2q-2}{3}}P_2\left(\frac{3k+1}{3q}\right)P_2\left(\frac{-3k+2q-1}{3q}\right)\\
=\sum_{k=0}^{\frac{2q-2}{3}}\left(\frac{18k^2+(12-18q)k+2-6q+3q^2}{18q^2}\right) \left( \frac{18k^2+(12-6q)k+2-2q-q^2}{18q^2}\right)\\
=\sum_{k=0}^{\frac{2q-2}{3}}\left( \frac{324k^4+(432-432q)k^3+(144q^2-432q+216)k^2}{324q^4}\right)\qquad \qquad \qquad \qquad\\
+\left(\frac{(96q^2-144q+48)k+(4-16q+16q^2-3q^4  )}{324q^4}\right). \qquad \qquad \qquad \qquad \\
\end{multline*}
Using the formula for sum of fixed powers of consecutive numbers we obtain
\begin{equation}\label{ES22}
\sum_{k=0}^{\frac{2q-2}{3}} P_2\left(\frac{3k+1}{3q}\right)P_2\left(\frac{-3k+2q-1}{3q}\right)=\frac{-26q^5-45q^4-160q^2+156q+120}{324 \times 45 \times q^4}.
\end{equation}
Now we evaluate the sum
\begin{multline*}
\sum_{k=\frac{2q+1}{3}}^{q-1}P_2\left(\frac{3k+1}{3q}\right)P_2\left(\frac{-3k+5q-1}{3q}\right)\\
=\sum_{k=\frac{2q+1}{3}}^{q-1}\left(\frac{18k^2+(12-18q)k+2-6q+3q^2}{18q^2}\right)\left( \frac{18k^2+(12-42q)k+2-14q+23q^2}{18q^2}\right)\\
=\sum_{k=\frac{2q+1}{3}}^{q-1}\left( \frac{324k^4+(432-1080q)k^3+(1224q^2-1080q+216)k^2}{324q^4}\right) \qquad\\+\sum_{k=\frac{2q+1}{3}}^{q-1}\left( \frac{(816q^2-540q^3-360q+48)k+(69q^4-180q^3+136q^2-40q+4)}{324q^4}\right).
\end{multline*}
From the formulas for sum of fixed powers of consecutive numbers we get
\begin{equation}\label{ES23}
\sum_{k=\frac{2q+1}{3}}^{q-1}P_2\left(\frac{3k+1}{3q}\right)P_2\left(\frac{-3k+5q-1}{3q}\right)=\frac{-13q^5+45q^4+90q^3-80q^2+78q-120}{324 \times 45 \times q^4}.
\end{equation}
Next we look at the sum
\begin{multline*}
\sum_{k=0}^{\frac{q-4}{3}}P_2\left(\frac{3k+2}{3q}\right)P_2\left(\frac{-3k+q-2}{3q}\right)\\
=\sum_{k=0}^{\frac{q-4}{3}}\left( \frac{18k^2+(24-18q)k+8-12q+3q^2}{18q^2}\right) \left( \frac{18k^2+(24+6q)k+8+4q-q^2}{18q^2} \right)\\
=\sum_{k=0}^{\frac{q-4}{3}}\left( \frac{324k^4+(864-216q)k^3+(864-432q-72q^2)k^2}{324q^4}\right) \quad\\
+\left(\frac{(36q^3-96q^2-288q+384)k+(64-64q-32q^2+24q^3-3q^4)}{324q^4}\right).
\end{multline*}
Using the formula for sum of fixed powers of consecutive numbers we get
\begin{equation}\label{ES24}
\sum_{k=0}^{\frac{q-4}{3}}P_2\left(\frac{3k+2}{3q}\right)P_2\left(\frac{-3k+q-2}{3q}\right)=\frac{-26q^5+90q^4+180q^3-160q^2+156q-240}{324 \times 90 \times q^4}.
\end{equation}
The last sum to be handled is
\begin{multline*}
\sum_{k=\frac{q-1}{3}}^{q-1}P_2\left(\frac{3k+2}{3q}\right)P_2\left(\frac{-3k+4q-2}{3q}\right)\\
=\sum_{k=\frac{q-1}{3}}^{q-1} \left( \frac{18k^2+(24-18q)k+8-12q+3q^2}{18q^2}\right) \left( \frac{18k^2+(24-30q)k+8-20q+11q^2}{18q^2} \right)\\
=\sum_{k=\frac{q-1}{3}}^{q-1}\left( \frac{324k^4+864(1-q)k^3+(729q^2-1728q+864)k^2}{324q^4}\right) \qquad \qquad \qquad \qquad \qquad \\
+\left(\frac{384-1152q+1056q^2-288q^3)k+(33q^4-192q^3+352q^2-256q+64)}{324q^4}\right). \quad 
\end{multline*}
This leads to
\begin{equation}\label{ES25}
\sum_{k=\frac{q-1}{3}}^{q-1}P_2\left(\frac{3k+2}{3q}\right)P_2\left(\frac{-3k+4q-2}{3q}\right)=\frac{-26q^5-45q^3-160q^2+156q+120}{324 \times 45 \times q^4}.
\end{equation}
Using equations (\ref{ES21}), (\ref{ES22}), (\ref{ES23}), (\ref{ES24}) and (\ref{ES25}) in equation (\ref{ES2}) we get
$$S^2(2q-1,3q)=\frac{q^4+330q^2-160q-6}{324 \times 15 \times q^3}.$$
Similar computations yield
$$S^3(2q-1,3q)=\frac{q^4+40q^3-160q^2+80q+4}{324 \times 10 \times q^3}   .$$
Using these values in (\ref{HLF}) gives
$$\zeta_K(-1,\mathfrak{Q})=\frac{1}{3^3\times 360}(q^3-6q^2+171q-166).$$
Since $q=9m+4$ we get
$$\zeta_K(-1,\mathfrak{Q})=\frac{1}{120}(9m^3+6m^2+19m+6).$$

\end{proof}

\begin{proof}
(Theorem \ref{C5}) We consider the quadratic field $K=\mathbb{Q}(\sqrt{9m^2+4m})$. Let $\mathfrak{C}$ and $\mathfrak{Q}$ denote the principle ideal class and the ideal class containing prime ideal $\mathfrak{q}$ respectively. By Proposition \ref{P1} and Proposition \ref{P6} we have
$$\zeta_K(-1,\mathfrak{C})=\zeta_K(-1,\mathfrak{Q}).$$
From a result of Lang (see section 4 of \cite{HL}) we conclude that $\mathfrak{C}=\mathfrak{Q}$. Consequently the ideal $\mathfrak{q} $ is principal. This gives 
$$\mathfrak{q}=<(x+\sqrt{D}y)/2 >\quad \mbox{ for some integers }x,y.$$
Taking norm of both the sides gives
$$4q=x^2-Dy^2.$$
This proves the theorem.
\end{proof}

%We remark that the inequality $(\ref{E1})$ is tight if and only if the polynomial $f(X)=X^2-3mX-m$ takes prime values for $X \in \{0,1, \ldots, \frac{3m-1}{2}\} \setminus \{m\}$ and both $m, 2m+1$ are prime. However, for $X=m$, $f(X)$ is not a prime whenever $m>1$. Combining this with Theorem \ref{DZ} and Proposition \ref{P1} we get the following corollary.
%
%\begin{cor}
%Let $m>1$ be an odd square-free integer. Then the class number of $K=\mathbb{Q}(\sqrt{9m^2+4m})$ is one if and only if $f(X)=X^2-3mX-m$ takes prime values for $X \in \{0,1, \ldots, \frac{3m-1}{2}\} \setminus \{m\}$ and both $m,2m+1$ are prime.
%\end{cor}
%
%In light of the result of Biro and Lapkova \cite{BL}, we conclude that for $m>...$ whenever both $m$ and $2m+1$ are primes the polynomial $X^2-3mx-m$ takes a composite value for some $X \in \{0,1, \ldots, \frac{3m_1}{2}\}$.\\
%
%Corollary \ref{CNT1} and Corollary \ref{CNT2} follows from the observation that $3$ is always a divisor of $s^2-3ms+m$ whenever $m \equiv 1 \pmod 3$ and $s^2 \equiv 1 \pmod 3$.\\

%{\bf Now we discuss some computation/conclusion.}
%
%Can we say that we can obtain all class number one fields in a family if the fundamental unit of the family is known?

We end with tables of real quadratic fields in the family $\mathbb{Q}(\sqrt{9m^2+4m})$ and their class number. We consider $m$ in the range $[-160,160]$ and divide them in cases depending on congruency class of $m$ modulo $3$ and consider only those $D$ which are square-free.\\

\newpage

\begin{table}[H]
%\parbox{00.32\textwidth}{
\parbox[t]{0.32\linewidth}{
\centering
	\caption{For $m\equiv0(\bmod3)$}\label{tab:1}
		\begin{tabular}[t]{|c|c|c|}
			\hline
			$m$ & $D$& $h_K$\\ \hline			
%			-195    &       341445  &24\\
%			-183    &       300669  &20\\
%			-177    &       281253  &16\\
%			-165    &       244365  &28\\
			-159    &       226893  &22\\
		-141    &       178365  &16\\
		-129    &       149253  &12\\
		-123    &       135669  &30\\
		-111    &       110445  &20\\
		-105    &       98805   &16\\
			-87     &       67773   &8\\
			-69     &       42573   &6\\
			-57     &       29013   &6\\
			-51     &       23205   &8\\
			-39     &       13533   &8\\
			-33     &       9669    &10\\
			-21     &       3885    &4\\
			-15     &       1965    &2\\
			-3      &       69      &1\\
		%	0       &       0       &0\\

			3       &       93      &1\\
			15      &       2085    &2\\
			21      &       4053    &4\\
			33      &       9933    &4\\
			39      &       13845   &4\\
			51      &       23613   &6\\
			57      &       29469   &12\\
			87      &       68469   &18\\
			105     &       99645   &16\\
		111     &       111333  &12\\
		123     &       136653  &16\\
		129     &       150285  &24\\
		141     &       179493  &12\\
			159     &       228165  &24\\
%			165     &       245685  &16\\
%			177     &       282669  &32\\
%			183     &       302133  &16\\
%			195     &       343005  &28\\
			\hline
		\end{tabular}
}
\hfill
%\begin{table}[H]	
\parbox[t]{0.32\linewidth}{
	\caption{For $m\equiv1(\bmod3)$}\label{tab:2}
	\centering	
%\caption{For $m\equiv 1 \pmod 3$.}	
		\begin{tabular}[t]{|c|c|c|}
			\hline
			$m$ & $D$& $h_K$\\ \hline
%			-197    &       348493  &20\\
%			-185    &       307285  &32\\
%			-179    &       287653  &31\\
%			-173    &       268669  &62\\
%			-167    &       250333  &21\\
			-155    &       215605  &28\\
			-149    &       199213  &28\\
			-143    &       183469  &30\\
			-137    &       168373  &24\\
			-119    &       126973  &20\\
			-113    &       114469  &32\\
			-107    &       102613  &24\\
			-101    &       91405   &16\\
			-95     &       80845   &20\\
			-89     &       70933   &24\\
			-83     &       61669   &23\\
			-77     &       53053   &16\\
			-71     &       45085   &14\\
			-65     &       37765   &16\\
			-59     &       31093   &10\\
			-53     &       25069   &16\\
			-47     &       19693   &11\\
			-41     &       14965   &8\\
			-35     &       10885   &10\\
			-29     &       7453    &6\\
			-23     &       4669    &8\\
			-17     &       2533    &4\\
			-11     &       1045    &4\\
			-5      &       205     &2\\
			1       &       13      &1\\
			7       &       469     &3\\
			31      &       8773    &7\\
			37      &       12469   &14\\
			43      &       16813   &10\\
			55      &       27445   &12\\
			61      &       33733   &10\\
			67      &       40669   &13\\
			73      &       48253   &12\\
			79      &       56485   &20\\
			85      &       65365   &20\\
			91      &       74893   &18\\
			97      &       85069   &24\\
			109     &       107365  &20\\
			115     &       119485  &30\\
			127     &       145669  &30\\
			133     &       159733  &24\\
			139     &       174445  &28\\
			145     &       189805  &24\\
			151     &       205813  &30\\
		157     &       222469  &52\\
%			163     &       239773  &31\\
%			181     &       295573  &30\\
%			187     &       315469  &48\\
%			193     &       336013  &32\\
%			199     &       357205  &48\\
			\hline
		\end{tabular}
}
\hfill
%	\label{tab:forpol}
%\quad
\parbox[t]{0.32\linewidth}{
%\begin{table}[H]
\caption{For $m\equiv2(\bmod3)$}\label{tab:3}
\centering
%\caption{For $m\equiv 2 \pmod 3$.}
	\begin{tabular}[t]{|c|c|c|}
		\hline
		$m$ & $D$& $h_K$\\ \hline
%		-199    &       355613  &13\\
%		-193    &       334469  &28\\
%		-187    &       313973  &20\\
%		-163    &       238469  &16\\
 	-157    &       221213  &14\\
	-151    &       204605  &16\\
		-145    &       188645  &16\\
		-139    &       173333  &10\\
		-133    &       158669  &18\\
		-127    &       144653  &12\\
		-115    &       118565  &10\\
		-109    &       106493  &14\\
		-103    &       95069   &14\\
		-97     &       84293   &8\\
		-91     &       74165   &8\\
		-85     &       64685   &12\\
		-79     &       55853   &6\\
		-73     &       47669   &8\\
		-67     &       40133   &9\\
		-61     &       33245   &8\\
		-55     &       27005   &8\\
		-43     &       16469   &5\\
		-37     &       12173   &4\\
		-19     &       3173    &3\\
		-13     &       1469    &2\\
		-7      &       413     &1\\
		-1      &       5       &1\\
		11      &       1133    &1\\
		17      &       2669    &4\\
		23      &       4853    &3\\
		29      &       7685    &4\\
		35      &       11165   &4\\
		41      &       15293   &4\\
		47      &       20069   &10\\
		53      &       25493   &4\\
		59      &       31565   &10\\
		65      &       38285   &12\\
		71      &       45653   &7\\
		77      &       53669   &8\\
		83      &       62333   &5\\
		89      &       71645   &8\\
		95      &       81605   &8\\
		101     &       92213   &14\\
		107     &       103469  &13\\
		113     &       115373  &8\\
		137     &       169469  &20\\
		143     &       184613  &14\\
		149     &       200405  &16\\
		155     &       216845  &12\\
%		161     &       233933  &12\\
%		167     &       251669  &16\\
%		173     &       270053  &18\\
%		179     &       289085  &16\\
%		185     &       308765  &16\\
%		191     &       329093  &15\\
%		197     &       350069  &18\\
		\hline
	\end{tabular}
}
	\end{table}

\newpage

\section*{Appendix}
Her we compute $S^3(2q-1,3q)$ using Lemma \ref{L4}.
\begin{multline}\label{P6E1}
S^3(2q-1,3q)=\sum_{j=0}^{3q-1}P_1(j/c)P_3((2q-1)j/c)\\
=\sum_{k=0}^{q-1}P_1((3k)/c)P_3((2q-1)(3k)/c)+\sum_{k=0}^{q-1}P_1((3k+1)/c)P_3((2q-1)(3k+1)/c)\\+\sum_{k=0}^{q-1}P_1((3k+2)/c)P_3((2q-1)(3k+2)/c)
\end{multline}
Since $q \equiv 1 \pmod 3$, we observe that
\begin{multline*}
\{(3k)(2q-1)/3q\}=(3q-3k)/3q \mbox{ for } k>0, \quad \{(3k+1)(2q-1)/3q\}=(-3k+2q-1)/3q \quad\\
 \mbox{ for }k \leq (2q-2)/3,
 \{(3k+1)(2q-1)/3q\}=(-3k+5q-1)/3q \mbox{ for }k \geq (2q+1)/3,\\
  \{(3k+2)(2q-1)/3q\}=(-3k+q-2)/3q \mbox{ for } k \leq (q-4)/3\\
 \mbox{ and } \{(3k+2)(2q-1)/3q\}=(-3k+4q-2)/3q \mbox{ for } k \geq (q-1)/3.
\end{multline*}
Recall that $P_i(x)=P_i(\{x\})$ for $i=1,2,3$ and $x \in \mathbb{R}$. Hence we obtain
\begin{multline}\label{ES2}
S^3(2q-1,3q)=\sum_{k=0}^{q-1} P_1\left(\frac{3k}{3q}\right)P_3\left(\frac{3q-3k}{3q}\right)+\sum_{k=0}^{\frac{2q-2}{3}}P_1\left(\frac{3k+1}{3q}\right)P_3\left(\frac{-3k+2q-1}{3q}\right)\\
+\sum_{k=\frac{2q+1}{3}}^{q-1}P_1\left(\frac{3k+1}{3q}\right)P_3\left(\frac{-3k+5q-1}{3q}\right)+\sum_{k=0}^{\frac{q-4}{3}}P_1\left(\frac{3k+2}{3q}\right)P_3\left(\frac{-3k+q-2}{3q}\right)\\
+\sum_{k=\frac{q-1}{3}}^{q-1}P_1\left(\frac{3k+2}{3q}\right)P_3\left(\frac{-3k+4q-2}{3q}\right).\\
\end{multline}
Now, we evaluate each sum individually.
\begin{equation*}
\begin{split}
\begin{aligned}[b]
\sum_{k=0}^{q-1} P_1\left(\frac{3k}{3q}\right)P_3\left(\frac{3q-3k}{3q}\right)
&=\sum_{k=0}^{q-1}\left(\frac{-3q+6k}{6q}\right) \left( \frac{-q^2k+3qk^2-2k^3}{2q^3}\right)\\
&=\sum_{k=0}^{q-1}\left( \frac{q^3k-5q^2k^2+8qk^3-4k^4}{4q^4}\right).
\end{aligned}
\end{split}
\end{equation*}
Now we use the formula for sum of fixed powers of consecutive numbers to obtain
\begin{equation}\label{ES21}
\sum_{k=0}^{q-1} P_1\left(\frac{3k}{3q}\right)P_3\left(\frac{3q-3k}{3q}\right)=\frac{q^4-5q^2+4}{120q^3}.
\end{equation}
Next, we have
\begin{equation*}
\begin{split}
\begin{aligned}[b]
&\sum_{k=0}^{\frac{2q-2}{3}}P_1\left(\frac{3k+1}{3q}\right)P_3\left(\frac{-3k+2q-1}{3q}\right) \\
&=\sum_{k=0}^{\frac{2q-2}{3}}   \left(\frac{6k+2-3q}{6q}\right)  \times \\ 
&\left( \frac{-54k^3+(27q-54)k^2+(9q^2+18q-18)k+(-2q^3+3q^2+3q-2)}{54q^3}\right) 
\\
&=\sum_{k=0}^{\frac{2q-2}{3}}\left( \frac{-324k^4+(324q-432)k^3+(-27q^2+324q-216)k^2}{324q^4}\right)\\
&+\left( \frac{(-39q^3-18q^2+108q-48)k+(6q^4-13q^3-3q^2+12q-4)}{324q^4}\right) 
\end{aligned}
\end{split}
\end{equation*}
Using the formula for sum of fixed powers of consecutive numbers we obtain
\begin{equation}\label{ES22}
\sum_{k=0}^{\frac{2q-2}{3}} P_1\left(\frac{3k+1}{3q}\right)P_3\left(\frac{-3k+2q-1}{3q}\right)=\frac{2q^5+15q^4-5q^3+50q^2-52q-40}{4860q^4}.
\end{equation}
Now we evaluate the sum
\begin{equation*}
\begin{split}
\begin{aligned}[b]
&\sum_{k=\frac{2q+1}{3}}^{q-1}P_1\left(\frac{3k+1}{3q}\right)P_3\left(\frac{-3k+5q-1}{3q}\right)\\
&=\sum_{k=\frac{2q+1}{3}}^{q-1}\left(\frac{6k+2-3q}{6q}\right) \times \\
& \left( \frac{-54k^3+(189q-54)k^2+(-207q^2+126q-18)k+(70q^3-69q^2+21q-2)}{54q^3}\right)\\
& =\sum_{k=\frac{2q+1}{3}}^{q-1}\left( \frac{-324k^4+(1296q-432)k^3+(-1809q^2+1296q-216)k^2}{324q^4}\right) \qquad\\+
& \left( \frac{(1041q^3-1206q^2+432q-48)k+(-210q^4+347q^3-201q^2+48q-4)}{324q^4}\right).
\end{aligned}
\end{split}
\end{equation*}
From the formulas for sum of fixed powers of consecutive numbers we get
\begin{equation}\label{ES23}
\sum_{k=\frac{2q+1}{3}}^{q-1}P_1\left(\frac{3k+1}{3q}\right)P_3\left(\frac{-3k+5q-1}{3q}\right)=\frac{-43q^5+30q^4-35q^3+20q^2-52q+80}{9720q^4}.
\end{equation}
Next we look at the sum
\begin{equation*}
\begin{split}
\begin{aligned}[b]
&\sum_{k=0}^{\frac{q-4}{3}}P_1\left(\frac{3k+2}{3q}\right)P_3\left(\frac{-3k+q-2}{3q}\right)\\
&=\sum_{k=0}^{\frac{q-4}{3}}\left( \frac{6k+4-3q}{6q}\right) \times \\
 &\left( \frac{-54k^3+(-27q-108)k^2+(9q^2-36q-72)k+(2q^3+6q^2-12q-16)}{54q^3} \right)\\
&=\sum_{k=0}^{\frac{q-4}{3}}  \left( \frac{-324k^4-864k^3+(135q^2-864)k^2}{324q^4}\right) \quad\\
&+\left(\frac{(-15q^3+180q^2-384)k+(-6q^4-10q^3+60q^2-64)}{324q^4}\right)
\end{aligned}
\end{split}
\end{equation*}
Using the formula for sum of fixed powers of consecutive numbers we get
\begin{equation}\label{ES24}
\sum_{k=0}^{\frac{q-4}{3}}P_1\left(\frac{3k+2}{3q}\right)P_3\left(\frac{-3k+q-2}{3q}\right)=\frac{-43q^5+30q^4-35q^3+20q^2-52q+80}{9720 q^4}.
\end{equation}
The last sum to be handled is
\begin{equation*}
\begin{split}
\begin{aligned}[b]
&\sum_{k=\frac{q-1}{3}}^{q-1}P_1\left(\frac{3k+2}{3q}\right)P_3\left(\frac{-3k+4q-2}{3q}\right)\\
&=\sum_{k=\frac{q-1}{3}}^{q-1}  \left( \frac{6k+4-3q}{6q}\right) \times \\ 
 &\left( \frac{-54k^3+(135q-108)k^2+(-99q^2+180q-72)k+(20q^3-66q^2+60q-16)}{54q^3} \right)\\
&=\sum_{k=\frac{q-1}{3}}^{q-1}\left( \frac{-324k^4+(972q-864)k^3+(-999q^2+1944q-864)k^2}{324q^4}\right) \qquad \qquad \qquad \qquad \qquad \\
&+\left(\frac{(417q^3-1332q^2+1296q-384)k+(278q^3-44q^2+288q-60q^4-64)}{324q^4}\right). \quad 
\end{aligned}
\end{split}
\end{equation*}
This leads to
\begin{equation}\label{ES25}
\sum_{k=\frac{q-1}{3}}^{q-1}P_1\left(\frac{3k+2}{3q}\right)P_3\left(\frac{-3k+4q-2}{3q}\right)=\frac{2q^5+15q^4-5q^3+50q^2-52q-40}{4860q^4}.
\end{equation}
Using equations (\ref{ES21}), (\ref{ES22}), (\ref{ES23}), (\ref{ES24}) and (\ref{ES25}) in equation (\ref{ES2}) we get
$$S^3(2q-1,3q)=\frac{q^4+40q^3-165q^2+80q+4}{3240q^3}.$$

%\noindent\textbf{Acknowledgements.}
%The second author would like to appreciate the hospitality provided by Harish-Chandra Research Institute, Allahabad and the discussions with algebraic number theory group there.% Also we are very much thankful to the anonymous referee.


\begin{thebibliography}
{100}
%\bibitem{AC55} N. C. Ankeny and S. Chowla, {\it On the divisibility of the class number of quadratic fields}, Pacific J. Math. {\bf 5} (1955), 321--324.

\bibitem{AB} A. Baker, Linear forms in the logarithms of algebraic numbers. I, II, III, {\em Mathematika} {\bf 13} (1966), 204-216; ibid. {\bf 14} (1967), 102-107; ibid {\bf 14} (1967), 220-228.

\bibitem{B1} A. Biro, Yokoi's conjecture, {\em Acta Arith.} {\bf 106} (2003), no. 1, 85-104.

\bibitem{B2} A. Biro, Chowla's conjecture, {\em Acta Arith.} {\bf 107} (2003), no. 2, 179-194.

\bibitem{BL} A. Biro, K. Lapkova The class number one problem for real quadratic fields $\mathbb{Q}(\sqrt{(an)^2+4a})$, {\em Acta Arith.} {\bf 172} (2016), no. 2, 117-131.

%\bibitem{BS01} Y. Bugeaud and T. N. Shorey, {\it On the number of solutions of the generalized Ramanujan-Nagell equation}, J. Reine Angew. Math. {\bf 539} (2001), 55--74.

\bibitem{BK1} D. Byeon, H. K. Kim, Class number 1 criteria for real quadratic fields of Richaud-Degert type, {\em J. Number Theory}, {\bf 57} (1996), no. 2, 328-339.

\bibitem{BK2}
	D. Byeon and H. K. Kim, Class number 2 criteria for real quadratic fields of Richaud-Degert type, {\em J. Number Theory} {\bf 62} (1997), 257--272.
	
 \bibitem{KAM}
    K. Chakraborty, A. Hoque and M. Mishra, A note on certain real quadratic fields with class number up to three, {\em Kyushu Journal of Mathematics} (To appear).	

%\bibitem{CL84} H. Cohen and H. W. Lenstra Jr., {\it Heuristics on class groups of number fields}, in: Number theory, Noordwijkerhout 1983, 33--62, 
%Lecture Notes in Mathematics, 1068, Springer, Berlin, 1984.

%\bibitem{Cohn1} J. H. E. Cohn, {\it Square Fibonacci numbers, etc.}, The Fibonacci Quarterly {\bf 2.2} (1964) 109--113.

%\bibitem{Cohn} J. H. E. Cohn, {\it On the class number of certain imaginary quadratic fields}, Proc. Amer. Math. Soc. {\bf 130} (2001), 1275--1277.

%\bibitem{CO03} J. H. E. Cohn, {\it On the Diophantine equation $x^n=Dy^2+1$}, Acta Arith. {\bf 106} (2003), 73--83.

\bibitem{GD} G. Degert, Uber die Bestimmung der Grundeinheit gewisser reell-quadratischer Zahlkorper, {\em Abh. Math. Sem. Univ. Hamburg} {\bf 22} (1958), 92-97.

%\bibitem{CM02} K. Chakraborty and M. R. Murty, {\it On the number of real quadratic fields with class number divisible by $3$}, Proc. Amer. Math. Soc., {\bf 131} (2003), 41-44.

%\bibitem{GR01} B. H. Gross and D. H. Rohrlich, {\it Some results on the Mordell-Weil group of the Jacobian of the Fermat curve}, Invent. Math. {\bf 44} (1978), 201--224.

%\bibitem{ES} J. H. Evertse and J. H. Silverman, {\it Uniform bounds for the number of solutions to $Y^n=f(X)$}, Math. Proc. Camb. Phil. Soc. {\bf 100} (1986), 237--248.

\bibitem{CFG}
	C. F. Gauss, {\em Disquisitiones arithmaticae.} Translated into English by Arthur A. Clarke, S. J. Yale University Press, New Haven, Conn., London.

\bibitem{HH}
    H. Hasse, \"Uber mehrklassige, aber eigen schlechtige reell-quadratische Zahlk\:orper, {\em Elem. Math.} {\bf 20} (1965), 49--59.

%\bibitem{HS15} A. Hoque and H. K. Saikia, {\it A family of imaginary quadratic fields whose class numbers are multiples of three}, J. Taibah Univ. Sci., {\bf 9} (2015), 399-402.

%\bibitem{HS16} A. Hoque and H. K. Saikia,
%{\it On the divisibility of class numbers of quadratic fields and the solvability of diophantine equations}, SeMA J. {\bf 73} (2016), 213--217.

%\bibitem{HS152} A. Hoque and H. K. Saikia {\it On generalized Mersenne primes and class-numbers of equivalent quadratic fields and cyclotomic fields}, SeMA J., {\bf 67} (2015), 71-75.

%\bibitem{IS11} K. Ishii, {\it On the divisibility of the class number of imaginary quadratic fields}, Proc. Japan Acad. {\bf 87}, Ser. A (2011), 142--143.

%\bibitem{IT11p} A. Ito, {\it A note on the divisibility of class numbers of imaginary quadratic fields $\mathbb{Q}(\sqrt{a^2-k^n})$}, Proc. Japan Acad. {\bf 87}, Ser. A (2011), 151--155.

%\bibitem{IT11} A. Ito, {\it Remarks on the divisibility of the class numbers of imaginary quadratic fields $\mathbb{Q}(\sqrt{2^{2k}-q^n})$}, Glasgow Math. J. {\bf 53} (2011), 379--389.

%\bibitem{IT15} A. Ito, {\it Notes on the divisibility of the class numbers of imaginary quadratic fields $\mathbb{Q}(\sqrt{3^{2e}-4k^n})$}, Abh. Math. Semin. Univ. Hamburg {\bf 85} (2015), 1--21.

%\bibitem{KI00} Y. Kishi, {\it A constructive approach to Spiegelung relations between 3-Ranks of absolute ideal class groups and congruent ones modulo $(3)^2$ in quadratic fields},
%J. Number Theory, {\bf 83} (2000), 1-49.

%\bibitem{KM00} Y. Kishi and K. Miyake, {\it parameterization of the quadratic fields whose class numbers are divisible by three}, J. Number Theory, {\bf 80} (2000), 209-217.

%\bibitem{KI09} Y. Kishi, {\it Note on the divisibility of the class number of certain imaginary quadratic fields}, Glasgow Math. J. {\bf 51} (2009), 187--191;
corrigendum, ibid. {\bf 52} (2010), 207--208.

\bibitem{KKST}
	 F. Kawamoto, Y. Kishi, H. Suzuki, K. Tomita, Real quadratic fields, continued fractions, and a construction of primary symmetric parts of ELE type, {\em Kyushu J. Math.} {\bf 73} (2019), 165-187.

\bibitem{HK}
    H. K. Kim, "A Conjecture of S. Chowla and Related Topics in Analytic Number Theory", Thesis, The Johns Hopkins Univ., Baltimore, 1988.

 \bibitem{HL}
    H. Lang, \"Uber eine Gattung elementar-arithmetischer Klassen invarianten reell-quadratischer Zahlk\"orper, {\em Angew. Math.} {\bf 223} (1968), 123--175.
    
\bibitem{JL} Lee, Jungyun, The complete determination of wide Richaud-Degert types which are not $5$ modulo $8$ with class number one, {\em Acta Arith.} {\bf 140} (2009), no. 1, 1-29.    

%\bibitem{KI10} Y. Kishi, {\it On the ideal class group of certain quadratic fields}, Glasgow Math. J., {\bf 52} (2010), 575-581. 

%\bibitem{LE50} V. A. Lebesgue, {\it Sur l'impossibilit\'{e} en nombres entiers de l'\'{e}quation $x^m = y^2 + 1$}, Nouvelles Annales des Math\'{e}matiques, {\bf 1} (1850), 178-181.

%\bibitem{LJ43} W. Ljunggren, {\it Some theorems on indeterminate equations of the form $\frac{x^n-1}{x-1}=y^q$}, Norsk Mat. Tidsskr. {\bf 25} (1943), 17--20.

%\bibitem{LO09} S. R. Louboutin, {\it On the divisibility of the class number of imaginary quadratic number fields}, Proc. Amer. Math. Soc. {\bf 137} (2009), 4025--4028.

\bibitem{SL}
	 S. R. Louboutin, Continued fractions and real quadratic field, {\em J. Number Theory} {\bf 30} (1988), 167-176.

\bibitem{RM}
	R. A. Mollin, Class number one criteria for real quadratic fields, {\em Proc. Japan Acad. Ser. A Math. Sci.} {\bf 63} 1987.

\bibitem{RM1}
	 R. A. Mollin, A simple criterion for solvability of both $x^2-dy^2=c$ and $x^2-dy^2=-c$, {\em New York J. Math.} {\bf 7} (2001), 87-97.

%\bibitem{RM97}M. R. Murty, {\it The ABC conjecture and exponents of class groups of quadratic fields}, Contemporary Math. {\bf 210} (1998), 85--95.

%\bibitem{RM99}M. R. Murty, {\it Exponents of class groups of quadratic fields}, Topics in number theory, Math. Appl. {\bf 467} Kluwer Acad. Publ., Dordrecht 1999, 229-239.

%\bibitem{NA22} T. Nagell, {\it \"{U}ber die Klassenzahl imagin\"{a}r quadratischer, Z\"{a}hlk\"{o}rper}, Abh. Math. Sem. Univ. Hamburg {\bf 1} (1922), 140--150.

%\bibitem{NA55} T. Nagell, {\it Contributions to the theory of a category of Diophantine equations of the second degree with two unknowns}, Nove Acta Regiae Soc. Sci. Uppsal. (4), {\bf 16} (1955), 1--38.

%\bibitem{CLS1} C. L. Siegel (under the pseudonym X). {\it The integer solutions of the equation $y^2=ax^n+bx^{n-1}+ \ldots +h$,} Gesammelte Abhandlungen, {\bf vol. 1} (Springer-Verlag, 1966), 207-208.


%\bibitem{CLS2} C. L. Siegel, {\it Uber einige Anwendungen diophantischer Approximationen (1929),} Gesammelte Abhandlungen, {\bf vol. 1} (Springer-Verlag, 1966), 209-266.

%\bibitem{WS} W. M. Schmidt, Equations over finite fields, An elementary approach,
%Lecture Notes in Mathematics, 536, Springer-Verlag, Berlin-New York, 1976.

\bibitem{CS}
	C. L. Siegel, Berechnung von Zetafunktionen an ganzzahligen Stellen, {\em Nachr. Akad. Wiss. G\"ottingen, Math.-Phys. Klasse} {\bf 10} (1969), 87--102.

%\bibitem{SO00} K. Soundararajan, {\it Divisibility of class numbers of imaginary quadratic fields}, J. London Math. Soc. {\bf 61} (2000), 681--690.

\bibitem{HMS} H. M. Stark,  A complete determination of the complex quadratic fields of class number one, {\em Michigan Math. J.} {\bf 14} (1967), 1-27.

%\bibitem{Th09} A. Thue, {\it \"{U}ber Ann\"{a}herungswerte algebraischerzehlen}, J. Reine Angew. Math. {\bf 135} (1909), 284--305.

%\bibitem{WE73} P. J. Weinberger, {\it Real Quadratic fields with Class number divisible by $n$}, J. Number theory {\bf 5} (1973), 237--241.

%\bibitem{YA70} Y. Yamamoto, {\it On unramified Galois extensions of quadratic number fields}, Osaka J. Math. {\bf 7} (1970), 57--76.

\bibitem{HY}
	H. Yokoi, On the Fundamental Unit of Real Quadratic Fields with Norm 1, {\em J. Number Theory} {\bf 2} (1970), 106--115.

\bibitem{HY1} H. Yokoi, On real quadratic fields containing units with norm $-1$, {\em Nagoya Math. J.} {\bf Vol. 33} (1968), 139-152.

\bibitem{DZ}
	D. B. Zagier, On the values at negative integers of the zeta function of a real quadratic field, {\em Enseig. Math.} {\bf 19} (1976), 55--95.

%\bibitem{MI12} M. Zhu and T. Wang, {\it The divisibility of the class number of the imaginary quadratic field $\mathbb{Q}(\sqrt{2^{2m}-k^n})$}, Glasgow Math. J. {\bf 54} (2012),
%149--154.

%\bibitem{IC03} H. Ichimura, {\it Note on the class numbers of certain real quadratic fields}, Abh. Math. Sem. Univ. Hamburg, {\bf 73} (2003), 281-288.
\end{thebibliography}
\end{document}